\documentclass[11pt]{article}
\usepackage[english]{babel}
\usepackage[applemac]{inputenc}
\usepackage{epsf}
\usepackage{color}
\usepackage[dvips]{graphicx}
\graphicspath{\\immagini}
\usepackage{amsmath,amsfonts,amssymb,amsthm}
\usepackage{placeins}
\newtheorem{theorem}{Theorem}

\newtheorem{proposition}{Proposition}

\newtheorem{corollary}{Corollary}

\begin{document}

\title{On the degree sequences of uniform hypergraphs}

\author{A.~Frosini~\thanks{Universit\`a
di Firenze, Dipartimento di Sistemi e Informatica, Viale Morgagni
65, 50134 Firenze, Italy}\and C.
Picouleau~\thanks{CEDRIC, CNAM, 292 rue St-Martin 75141, Paris cedex
03, France}
\and S.~Rinaldi~\thanks{Universit\`a di
Siena, Dipartimento di  Ingegneria dell'informazione e scienze
matematiche, Pian dei Mantellini 44, 53100 Siena, Italy}}
\maketitle

\begin{abstract}

In hypergraph theory, determining a characterization  of the degree sequence $d=(d_1,d_2,\ldots,d_n)$ where $d_1\ge d_2\ge\ldots,d_n$ are positive integers,
of an $h$-uniform simple hypergraph $\cal H$, and
deciding the complexity status of the reconstruction  of $\cal H$
from $d$,  are  two challenging open problems. They can be
formulated in the context of discrete tomography: asks whether
there is a matrix $A$ with positive projection vectors
$H=(h,h,\ldots,h)$ and $V=(d_1,d_2,\ldots,d_n)$ with distinct
rows.

In this paper we consider the two subcases where the vector 
$V$ is an homogeneous vector, and where $V$ is almost  homogeneous, i.e., $d_1-d_n=1$. We give a simple  characterization for these two subcases, and we show how to solve the related
 reconstruction problems in polynomial time. To
reach our goal, we use the concepts of Lyndon words and necklaces
of fixed density, and we apply some already known algorithms for
their efficient generation.

keywords: Discrete Tomography, Reconstruction problem, Lyndon
word, Necklace, Hypergraph degree sequence, Regular bipartite graph.
\end{abstract}

\section{Introduction}\label{intro}
The degree sequence, also called graphic sequence, of a simple
graph (a graph without loop or parallel edges) is the list of
vertex degrees, usually written in nonincreasing order, as $d=(
d_1,d_2,\ldots,d_n),d_1\ge d_2\ge\cdots\ge d_n$. The problem of
characterizing the graphic sequences of graphs was solved by Erd\"
os and Gallai (see \cite{Berge}):

\begin{theorem} {\bf (Erd\" os, Gallai)}
A sequence $d=(d_1,d_2,\ldots,d_n)$ where $d_1\ge d_2\ge\cdots\ge
d_n$ is graphic if and only if $\Sigma_{i=1}^nd_i$ is even and
$$\Sigma_{i=1}^k d_i\le k(k-1)+\Sigma_{i=k+1}^n\min\{k,d_i\}, 1\le k\le n. $$
\end{theorem}

An hypergraph ${\cal H}=(Vert,{\cal E})$ is defined as follows (see
\cite{Bergehyper}):  $Vert=\{v_1,\ldots,v_n\}$ is  a ground set of
vertices and ${\cal E}\subset 2^{\vert Vert
\vert}\setminus\emptyset$ is the set of hyperedges such that $e\not
\subset e'$ for any pair $e,e'$ of $\cal E$. The degree of a vertex
$v\in Vert$ is the number of hyperedges $e\in \cal E$ such that
$v\in e$.  An hypergraph ${\cal H}=(Vert,{\cal E})$ is $h$-uniform
if $\vert e\vert=h$ for all hyperedge $e\in \cal E$. Moreover ${\cal
H}=(Vert,{\cal E})$ has no parallel hyperedges, i.e., $e\ne e'$ for
any pair $e,e'$ of hyperedges. Thus a simple graph (loopless and
without parallel edges) is a $2$-uniform hypergraph.

%{\color{red}add the generalization results of the final part}

The problem of the characterization of the degree sequences of
$h$-uniform hypergraphs is one of the most relevant among the
unsolved problems in the theory of hypergraphs \cite{Bergehyper}
even for the case of $3$-uniform hypergraphs. For its last case Kocay and Li show that any two $3$-uniform hypergraphs with the same degree sequence can be transformed into each other using a sequence of trades \cite{kocay}. Furthermore the complexity
status of the reconstruction problem is still open.

This problem has been related to a class of problems that are of
great relevance in the field of discrete tomography. More
precisely the aim of discrete tomography is the retrieval of
geometrical information about a physical structure, regarded as a
finite set of points in the integer lattice, from measurements,
generically known as projections, of the number of atoms in the
structure that lie on lines with fixed scopes. A common
simplification is to represent a finite physical structure as a
binary matrix, where an entry is $1$ or $0$ according to the
presence or absence of an atom in the structure at the
corresponding point of the lattice. One of the challenging
problems in the field is then to reconstruct the structure, or, at
least, to detect some of its geometrical properties from a small
number of projections. One can refer to the books of G.T. Herman
and A. Kuba \cite{Herman,Herman2} for further information on the
theory, algorithms and applications of this classical problem in
discrete tomography.

Here we recall the seminal result in the field of the discrete
tomography due to Ryser~\cite{ryser}. Let
$H=(h_1,\ldots,h_m),h_1\ge h_2\ge\cdots\ge h_m,$ and
$V=(v_1,\ldots,v_n),v_1\ge v_2\ge\cdots\ge v_n,$ be two
nonnegative integral vectors, and ${\cal U}(H,V)$ be the class of
binary matrices $A=(a_{ij})$ satisfying
\begin{eqnarray}
\Sigma_{j=1}^na_{ij}=&h_i&1\le i\le m\\
\Sigma_{i=1}^ma_{ij}=&v_j&1\le j\le n
\end{eqnarray}
In this context  $H$ and $V$ are called the row, respectively
column, projection  of  $A$, as depicted in Fig.~\ref{set-matrix}.
Denoting by ${\bar V}=({\bar v_1},{\bar v_2},\ldots)$ the conjugate
sequence, also called the Ferrer sequence, of $V$ where $\bar
v_i=\vert\{v_j: v_j\in V, v_j\ge i\}\vert$. Ryser gave the
following \cite{ryser}:

\begin{theorem} {\bf (Ryser)} ${\cal U}(H,V)$ is nonempty if and only if
\begin{eqnarray}
\Sigma_{i=1}^mh_i&=&\Sigma_{i=1}^nv_i\\
\Sigma_{j=1}^ih_j&\ge&\Sigma_{j=1}^i{\bar v_j}\hskip5mm \forall
i\in\{1,\ldots,m\}
\end{eqnarray}
\end{theorem}

Moreover this characterization, and the reconstruction of $A$ from
its two projections $H$ and $V$, can be done in polynomial time (see
\cite{Herman}). Some applications in discrete tomography requiring
additional constraints can be found in
\cite{balogh,bro2col,batenburg,gardner,irving,vardi,prause,shlif}.

As shown in \cite{Berge} this problem is equivalent to the
reconstruction of a bipartite graph $G=(H,V,E)$ from its degree
sequences $H=(h_1,\ldots,h_m)$ and $V=(v_1,\ldots,v_n)$. Numerous
papers give some generalizations of this problem for the graphs
with colored edges (see \cite{pic,fixrows,3colnp,gale,guinrz}).

So, in this context, the problem of the characterization of the
degree sequence $(d_1,d_2,\ldots,d_n)$ of an $h$-uniform hypergraph
$\cal H$ (without parallel edges) asks whether there is a binary
matrix $A\in {\cal U}(H,V)$ with nonnegative projection vectors
$H=(h,h,\ldots,h)$ and $V=(d_1,d_2,\ldots,d_n)$ with distinct rows,
i.e., $A$ is the incidence matrix of $\cal H$ where rows and columns
correspond to hyperedges and vertices, respectively. To our
knowledge the problem of the reconstruction of a binary matrix with
distinct rows has not been studied in discrete tomography.

In this paper, we carry on our analysis in the special case where
the $h$-uniform hypergraph to reconstruct is also $d$-regular,
i.e., each vertex $v$ has the same degree $d$, in other words the
vector of the vertical projection is homogeneous, i.e.,
$V=(d,\ldots,d)$. We also study the problem where the $h$-uniform
hypergraph to reconstruct is almost $d$-regular, i.e.,
$V=(d,\ldots,d,d-1,\ldots,d-1)$, in other words the hypergraph has
span one.

We focus both on the decision problem, and on the related
reconstruction problem, i.e., the problem of determining the
existence of an element of ${\cal U}(H,V)$ consistent with $H$ and
$V$, and in affirmative case, how to quickly reconstruct it. To
accomplish these tasks, we will design an algorithm that runs in
polynomial time with respect to the dimensions $m$ and $n$ of the
matrix to reconstruct. The algorithm relies on the concepts of
Lyndon words and necklaces of fixed density, and uses an already
known algorithm for their efficient generation.

\section{Definitions and introduction of the problems}

Let $A$ be a binary matrix having $m$ rows and $n$ columns, and
let us consider the two integer vectors $H=(h_1,\dots ,h_m)$ and
$V=(v_1,\dots ,v_n)$ of its horizontal and vertical projections,
respectively, as defined in Section~\ref{intro} (see
Fig.~\ref{set-matrix}).

\begin{figure}
\begin{center}
\includegraphics[width = 4.5cm]{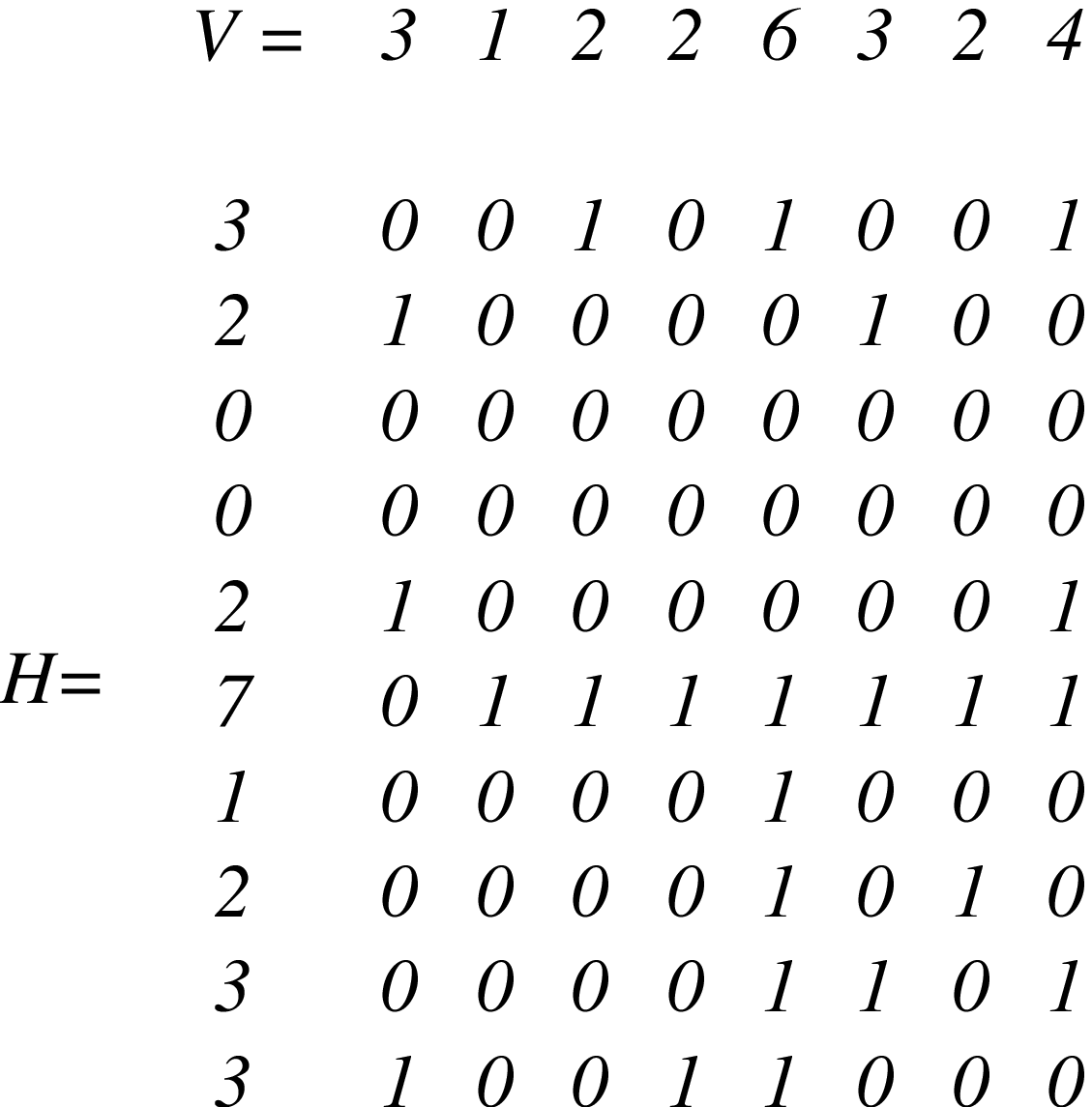}
\end{center}
\caption{A binary matrix, used in discrete tomography to represent
finite discrete sets, and its vectors $H$ and $V$ of horizontal
and vertical projections, respectively.} \label{set-matrix}
\end{figure}

In this paper we will consider some specialized versions of the
following general problems:\\

{\sc \bf Consistency} $(H,V,\mathcal{C})$
\begin{description}
\item{\bf Input:} two integer vectors $H$ and $V$, and a class of
discrete sets $\mathcal{C}$.

\item{\bf Question:} does there exist an element of $\mathcal{C}$
whose horizontal and vertical projections are $H$ and $V$,
respectively?
\end{description}

\medskip

{\sc \bf Reconstruction} $(H,V,\mathcal{C})$
\begin{description}
\item{\bf Input:} two integer vectors $H$ and $V$, and a class of
discrete sets.

\item{\bf Task:} reconstruct a matrix $A \in \mathcal{C}$ whose
horizontal and vertical projections are $H$ and $V$, respectively,
if it exists, otherwise give failure.
\end{description}

In \cite{rys}, Ryser gave a characterization of the instances of
$Consistency (H,V,\mathcal{C})$, with $\mathcal{C}$ being the class
of the binary matrices, that admit a positive answer. He moved from
the following trivial conditions that are necessary for the
existence of a matrix consistent with two generic vectors $H$ and
$V$ of projections:

\begin{description}
\item{{\bf Condition $1$}:} for each $1\leq i \leq m$ and $1\leq j \leq
n$, it holds $h_i \leq n$ and $v_j\leq m$;
 \item{{\bf Condition $2$}:}
$\Sigma_{i=1}^m h_i = \Sigma_{j=1}^n v_j,$
\end{description}

and then he added a third one to obtain the characterization, as
recalled in the Introduction.

The authors of \cite{bro}, pointed out that these two conditions are
also sufficient in case of homogeneous horizontal and vertical
projections, by showing their maximality w.r.t. the cardinality of
the related sets of solutions.

Ryser defined a well known greedy algorithm to solve
$Reconstruction(H,V,\mathcal{C})$ that does not compare the
obtained rows, and does not admit an easy generalization to
perform this further task.

In the sequel, we are going to consider the class of binary matrices
having no equal rows and homogeneous horizontal projections, due to
its connections, as mentioned in the Introduction, with the
characterization of the degree sequences of $h$-uniform hypergraphs.
Among them, we restrict our analysis to those matrices 
that are also, first, $d$-regular, i.e., whose vertical
projections are also homogeneous: $H=(h,\ldots,h)$ and $V=(v,\ldots,v)$; we denote this class by
$\mathcal{E}$; and, second, almost $d$-regular, i.e., whose vertical
projections are also almost homogeneous: $H=(h,\ldots,h)$ and $V=(v,\ldots,v,v-1,\ldots,v-1)$; we denote this class by
$\mathcal{E}_1$.\\

Now, we state a third necessary condition for answering to
$Consistency(H,V,\mathcal{E})$ (and also to $Consistency(H,V,\mathcal{E}_1)$ as we will see in Section \ref{spanone}):

\begin{description}
%\item{Condition $3$:} $Consistency(H,V,\mathcal{E})$ has a
%negative answer if the dimension of the vector $H$ is greater than
%${n}\choose{h}$.
%\item{Condition $3$:} if $Consistency(H,V,\mathcal{E})$ has
%positive answer, then the maximum value of the vertical
%projections $v$ is weakly upper bounded by $h/n \cdot
%{{n}\choose{h}}$, i.e., $v\leq h/n \cdot{{n}\choose{h}}$.
\item{{\bf Condition $3$}:} If $Consistency(H,V,\mathcal{E})$ has
positive answer, then  $$v\leq h/n \cdot{{n}\choose{h}}.$$
\end{description}

Condition $3$ can be rephrased, in our setting, as follows: there
does not exist a matrix having $H=(h,\ldots,h)$ and $V=(v,\ldots,v)$ as homogeneous
projections, and more than ${n}\choose{h}$ different rows; otherwise at least two rows will be identical. We will
prove that the three conditions 1,2, and 3 are also sufficient to solve (in
linear time) the problem $Consistency(H, V, \mathcal{E})$.

To this aim, we use an approach different from those standardly used
in Discrete Tomography: we consider each row of a matrix in
$\mathcal{E}$ as a binary word, and we group them into equivalence
classes according to their cyclic shifts, as defined in the next
section.

\section{The problem $Consistency(H,V,\mathcal{E})$}

Let us consider each row of a binary matrix as a binary finite word
$u=u_1 \: u_2 \: \dots \: u_n$, whose length $n$ is the number of
columns of the matrix, and whose number $h$ of $1$-elements is the
value of the horizontal projection.

We note that applying a cyclic shift to the word $u$, denoted by
$s(u)$, we obtain a different word $s(u)=u_2 \: u_3 \: \dots \: u_n
\: u_1$, unless the cases $u=(1)^n$ or $u=(0)^n$, of the same
length, and having the same number of $1$-elements inside. Iterating
the shift of a word $u$, we obtain a sequence of different words
that row wise arranged as a matrix, belong to $\mathcal{E}$. We
indicate with $s^k(u)$, where $k\geq 0$, the application of $k$
times the shift operator to the word $u$.

Unfortunately the words repeat after at most $n$ shifts, and
consequently the vertical projections of the obtained matrix are
upper bounded by $n$, so, in general, only a submatrix of a solution
of $Reconstruction(H,V,\mathcal{E})$ is achieved (see Fig.
\ref{matrix}).
%So, the elements of $(u)$ can be considered as possible different
%rows of a solution of $Reconstruction(H,V,\mathcal{E})$, so they
%can be arranged inside a solution matrix. Unfortunately, two
%problems may occur: there at most $n$ words conjugated with $u$,
%and using them, the vertical projections may not reach the desired
%value $v$.
The following trivial result holds:

\begin{proposition}\label{prop1}
Let $u$ be a binary word of length $n$ having $h\leq n$ $1$-elements
inside. Let us consider the $n \times n$ matrix $A$ obtained by row
wise arranging the $n$ cyclic shifts of $u$. Then, $A$ has the
horizontal and vertical projections equal to $h$.
\end{proposition}

As already noticed, the rows of the matrix $A$ may not all be
different. Throughout the paper we will denote by $M(u)$ the matrix
obtained by row wise arranging all the different cyclic shifts of a
word $u$. To establish how many different rows can be obtained by
shifting a given binary word, we need to recall the definitions and
main properties of necklaces and Lyndon words.

\medskip

Following the notation in~\cite{rus}, a {\em binary necklace}
(briefly {\em necklace}) is an equivalence class of binary words
under cyclic shift. We identify a necklace with the
lexicographically least representative $u$ in its equivalence class,
denoted by $[u]$. The set of all (the words representative of) the
necklaces with length $n$ is denoted $N(n)$. For example, $$N(4) =
\{ 0000, 0001, 0011, 0101, 0111, 1111 \} \, .
$$ An important class of necklaces are those that are aperiodic.
An aperiodic (i.e. period $\geq n$) necklace is called a {\em
Lyndon word}. Let $L(n)$ denote the set of all Lyndon words with
length $n$. For example, $L(4) = \{0001, 0011, 0111\}$.

We denote fixed-density necklaces, and Lyndon words in a similar
manner by adding the parameter $d$ to represent the number of
$1$-elements in the words. We refer to the number $d$ as the density
of the word. Thus the set of necklaces with density $d$ is
represented by $N(n, d)$, and the set of Lyndon words with density
$d$ is represented by $L(n, d)$. For example, $N(4,2) = \{ 0011,
0101 \}$, and $L(4,2)=\{ 0011 \}$.

It is known from Gilbert and Riordan \cite{gil} that the number of
fixed density necklaces and Lyndon words is
$$ N(n,d)=\frac{1}{n} \sum_{j\backslash \gcd({n,d})} \phi (j) {{n/j} \choose {d / j}} \,
, \qquad L(n,d)=\frac{1}{n} \sum_{j\backslash \gcd({n,d})} \mu (j)
{{n/j} \choose {d / j}} \, $$ respectively, where the symbols
$\phi$ and $\mu$ refer to the Euler and M\"obius functions.

Now we enlighten the connection between these objects and our
problem, refining Proposition \ref{prop1}:

\begin{proposition}
If $u$ is a word of length $n$ and density $h\leq n$, then the
cardinality of $[u]$ (i.e. the number of rows of $M(u)$) is a
divisor of $n$.
\end{proposition}

As a consequence, we have:

\begin{corollary}
If $u$ is a Lyndon word of length $n$ and density $h$, then the
cardinality of $[u]$, i.e. the number of rows of $M(u)$, is equal to
$n$, and the vertical projections of $M(u)$ are all equal to $h$.
\end{corollary}

The first $12$ rows of the matrix in Fig.~\ref{matrix} are obtained
by row wise arranging the $12$ different cyclic shifts of the Lyndon
word $u=(0)^6(1)^6$. Such a submatrix $M(u)$ has horizontal and
vertical projections equal to $6$, that is the density of $u$.

\begin{figure}
\begin{center}
\includegraphics[width = 8cm]{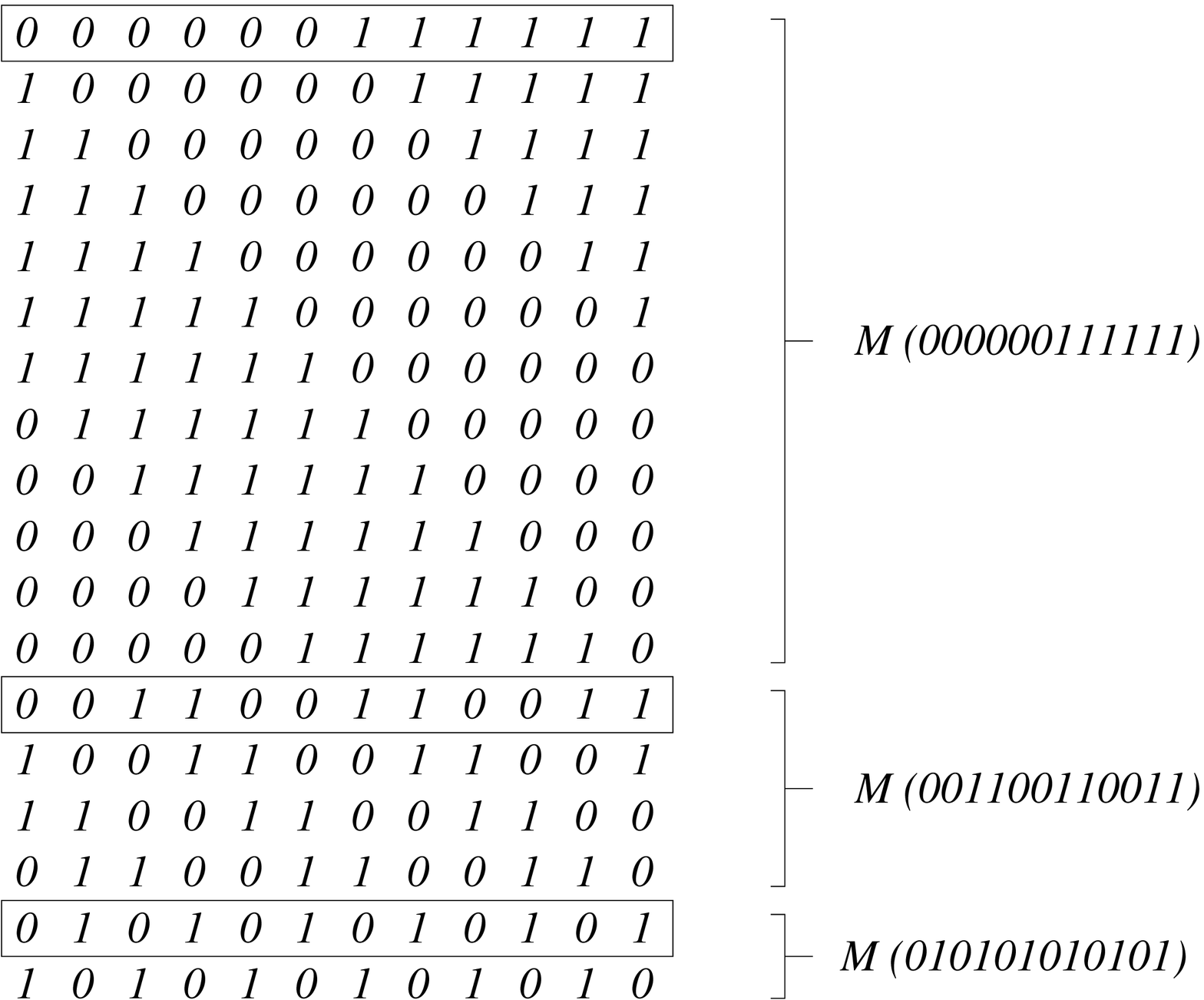}
\end{center}
\caption{A solution to $Reconstruction (H,V,{\mathcal{E}})$ when the
horizontal projections have constant value $6$, and the vertical
projections $9$. The submatrices $M(u)$ obtained by row wise
arranging the elements of three necklaces are highlighted. Note that
$M((0011)^3)$ and $M((01)^6)$ are the only two possible necklaces of
length $12$ and density $6$ having $4$ and $2$ rows, respectively.}
\label{matrix}
\end{figure}

\begin{proposition}
If $u=v^k$ (i.e. $u=v\dots v$, $k$ times), with $k=\gcd\{n,h\}$,
is a necklace of length $n$ and density $h$, and $v$ a Lyndon
word, then the cardinality of $[u]$ is equal to $n/k$, and the
vertical projections of $M(u)$ are all equal to $h/k$.
\end{proposition}

Figure \ref{matrix} shows the $12/3=4$ different cyclic shifts of
the word $u=(0011)^3$ arranged from row $13$ to row $16$ of the
matrix, and the $12/6 = 2$ different cyclic shifts of the word $v =
(01)^6$ in rows 17 and 18. All the rows of $M(u)$ have horizontal
projections equal to $6$ and vertical projections equal to $6/3 =
2$, while the rows of $M(v)$ have horizontal projections equal to
$6$ and vertical projections equal to $6/6 = 1$.

In the following we will prove that a pair $H$ and $V$ of
projections satisfy Conditions $1$, $2$, and $3$ if and only if
they are consistent with a matrix in $\mathcal{E}$, solving
$Consistency(H, V, \mathcal{E})$.

Let $d_0=1,d_1, d_2,\dots,d_t$ be the increasing sequence of the
common divisors of $n$ and $h$. The following equation holds:

$$
{{n}\choose{h}} = \sum_{i=0 \dots,t} \frac{n}{d_i} \:\:
L\left(\frac{n}{d_i},\frac{h}{d_i}\right ).$$

This equation is an immediate consequence of the fact that each
word of length $n$ and density $h$ belongs to exactly one
necklace.

\begin{theorem}\label{lyndon}
Let $H$ and $V$ be two homogeneous vectors of projections of
dimension $m$ and $n$, and elements $h$ and $v$, respectively,
satisfying Conditions $1$, $2$, and $3$, i.e., being a valid
instance of $Consistency(H,V,\mathcal{E})$. Then, there exists a
Lyndon word $L(n/d_i,h/d_i)$ such that $n/d_i\leq m$.
\end{theorem}

\proof

Let us proceed by contradiction assuming that there does not exist
a Lyndon word whose length is $n/d < m$, for each $d\in d_1,\dots
,d_t$. Since $H$ and $V$ are homogeneous, and satisfy Conditions
$1$ and $2$, then there exists a matrix $A$ having $H$ and $V$ as
projections (a consequence of Ryser's characterization of solvable
instances, as stated in \cite{bro}, Theorem $3$).

Let us assume that $d=d_t=\gcd\{n,h\}$, $h'=h/d$, and $n'=n/d$; from
Condition $1$, it holds
\begin{equation}\label{eq1}h' m = v n'
\end{equation}
with $n'$ and $h'$ coprime, so $v=h' (m/n')$, and $n'$ divides
$m$. The hypothesis $n'>m$ leads to a contradiction. \qed

Theorem~\ref{lyndon} can be rephrased saying that if $H$ and $V$
are homogeneous consistent vectors of projections, then there
exists a solution that contains all the elements of a necklace
$[u]$. The solution in linear time of
$Consistency(H,V,\mathcal{E})$ is a neat consequence:

\begin{corollary}
Let $H$ and $V$ be two homogeneous vectors satisfying Conditions
$1$, $2$, and $3$. There always exists a matrix having different
rows, and $H$ and $V$ as projections.
\end{corollary}

The result of Theorem~\ref{lyndon}, together with the following
proposition that point out a property of the necklace whose
representant is $u=(0)^{n-h}(1)^{h}$, will be used in the next
section to solve $Reconstruction(H,V,\mathcal{E})$.

\begin{proposition}\label{bar}
Let $u'$ be an element of the class $[u]$, with
$u=(0)^{n-h}(1)^{h}$. The elements $u'$, $s^h(u')$, $s^{2h}(u')$,
$\dots$, $s^{(k-1)h}(u')$, with $k= n/ \gcd \{n,h\}$, forms a
subclass of $[u]$, and they can be arranged in a matrix $A'$ such
that

\begin{description}
\item{1.} the vertical projections of $A'$ are homogeneous and
equal to $h / \gcd\{n,h\}$;

\item{2.} $A'$ is minimal with respect to the number of rows among
the matrices having $H$ as horizontal projections, and homogeneous
vertical projections.
\end{description}
\end{proposition}

The proof directly follows from the properties of the greatest
common divisor. Let us denote with $M_0(u)$ the matrix $A'$
defined in Proposition~\ref{bar}, and with $M_{i}(u)$ the matrix
defined in the same way starting from the word $u=
(1)^i(0)^{n-h}(1)^{h-i}$, with $0\leq i < \gcd\{n,h\}(=n/k)$.

\section{An algorithm to solve $Reconstruction(H,V,\mathcal{E})$}

We start recalling that in~\cite{sawada} a constant amortized time
(CAT) algorithm {\bf FastFixedContent} for the exhaustive generation
of necklaces $N(n,h)$ of fixed length and density is presented. The
author then shows that a slight modification of his algorithm can
also be applied for the CAT generation of the Lyndon words $L(n,h)$.
In particular, his algorithm --here denoted {\bf GenLyndon$(n,h)$}--
constructs a generating tree of the words, and since the tree has
height $h$, the computational cost of generating $k$ words of
$L(n,h)$ is $O( k \cdot h \cdot n)$.

\smallskip

Let us consider the following algorithm that reconstructs an element
of $\mathcal{E}$ from a couple of homogeneous horizontal and
vertical projections $H$ and $V$:

\smallskip

\noindent{\bf Rec$(H,V,{\mathcal E)}$}

\begin{description}
\item[Input :] Two homogeneous vectors: $H=(h, \ldots ,h)$ of
length $m$, and $V=(v, \ldots ,v)$ of length $n$, satisfying
Conditions $1$, $2$, and $3$. \item[Output :] An element of the
class $\mathcal E$ having $H$ and $V$ as horizontal and vertical
projections, respectively.

\item[Step 1:] Let compute the sequence $d_0=1 < d_1 < d_2 < \dots
< d_t$ of the common divisors of $n$ and $h$, and initialize the
matrix $A_{-1}=\emptyset$.

\item[Step 2:] For $i=0$ to $t$ do:

\begin{description}
\item[Step 2.1:] By applying {\bf GenLyndon$(n,h)$}, generate the
sequence of $q=\min{ \{ \lfloor v/h \rfloor, L(n,h) \} }$ Lyndon
words, denoted $u_1, \ldots ,u_q$. If $q\not= L(n,h)$, then do not
include in the sequence the Lyndon word $(0)^{n-h}(1)^h$.

\item[Step 2.2:] Create the matrix $A_i$, obtained by row wise
arranging the matrices $A_{i-1}$ and $M((u_j)^{d_i})$, for $j=1,
\ldots ,q$.

Update $v=v-q\cdot h$.

If $v=0$ then output $A_i$,

else if $q\not= L(n,h)$, create the matrix $A$ obtained by row wise
arranging the matrix $A_{i}$ with the column wise arranging of $d_i$
times the matrices $M(u)_j$, with $u=(0)^{n-h}(1)^h$, $j=0, \ldots
,q'-1$, and $q'=v \cdot \gcd\{n,h\}/h$,

else update $n=n/d_{i+1}$, and $h=h/d_{i+1}$.
\end{description}
\end{description}

\medskip

A brief explanation of Step 2.2 is needed: for each common divisor
$d_i$ of $n$ and $h$, the algorithm considers all the Lyndon words
of $L(n/d_i,h/d_i)$; if the matrices obtained from them can be
stuffed inside the solution matrix, then the algorithm performs this
action, and starts again the step with $i=i+1$, otherwise the
algorithm sets aside the word $u=(0)^{n/d_i}(1)^{h/d_i}$, and stuffs
the matrices obtained from the other Lyndon words in the solution
matrix. Since the remaining vertical projections $v'$ are less than
$h/d_i$, then the matrices $M(u)_j$, with $1\leq j \leq q'$ as
defined in Proposition \ref{bar}, can be used to fill the gap,
without going on generating the elements of
$L(n/d_{i+1},h/d_{i+1})$.

To better understand the reconstruction algorithm, we first propose
a simple example with the instance $H=(2, \ldots ,2)$ of length
$m=15$, and $V=(5, \ldots ,5)$ of length $n=6$. In Step $1$ the
values $d_0=1$, and $d_1=2$ are set.

In Step $2$, {\bf GenLyndon$(6,2)$} generates $q=2$ Lyndon words,
i.e. the words $000011$, and $000101$; since $L(6,2)=2$, then the
word $000011$ is included in the sequence. Now the matrix $A_0$,
depicted in Fig.~\ref{matrix2}, on the left, is created. Finally,
the values $v=5-2\cdot 2=1$, $n=6/2=3$, and $h=2/2=1$ are updated.

The second run of Step $2$ starts, and {\bf GenLyndon$(3,1)$}
generates the Lyndon word $001$. The final matrix $A_1$ is created
by row wise arranging $A_0$ with the matrix $M((001)^2)$ as shown in
Fig.~\ref{matrix2}, on the right.

\begin{figure}
\begin{center}
\includegraphics[width = 7.5cm]{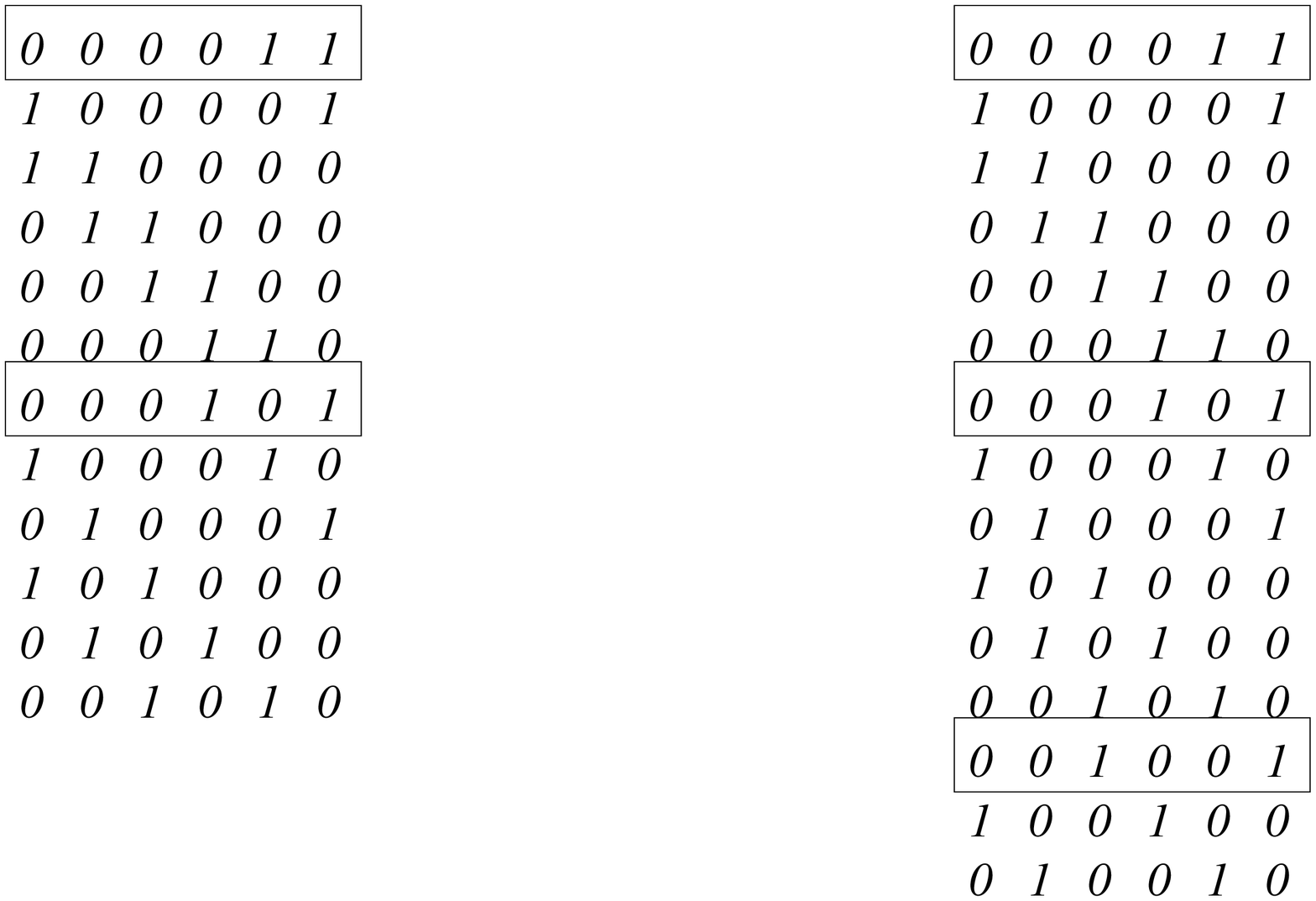}
\end{center}
\caption{The solution of dimension $15 \times 6$ obtained by
applying $Rec(H,V,{\mathcal{E}})$ when the horizontal projections
have constant value $2$, and the vertical projections $5$.}
\label{matrix2}
\end{figure}

A second example concerns the use of the word $(0)^{n-h}(1)^h$
that in certain cases is set aside from the sequence of Lyndon
words generated in Step $2$: the instance we consider is $H=(3,
\ldots ,3)$ of length $m=15$, and $V=(5, \ldots ,5)$ of length
$n=9$. In Step $1$ the values $d_0=1$, and $d_1=3$ are set.

In Step $2$, {\bf GenLyndon$(9,3)$} generates $q=\min\{\lfloor 5/3
\rfloor,L(9,3)\}=1$ Lyndon words, i.e. the word $000001011$; since
$q\not= L(9,3)$, then the word $000000111$ is not included in the
sequence. Now the matrix $A_0$, depicted in Fig.~\ref{matrix3}, on
the left, is created. The value $v=5-3\cdot 1=2$ is set.

Now, since $q\not= L(9,3)$, $q'=2(= 2 \cdot \gcd\{9,3\} / 3)$
submatrices of $M(000000111)$ are computed, as defined in
Proposition~\ref{bar}, and row wise arranged with $A_0$, obtaining
the matrix in Fig.~\ref{matrix3}, on the right.

Note that without the use of the Lyndon word $000000111$, the
procedure is not able to reach the solution since in the second
run of Step $2$, {\bf GenLyndon$(3,1)$} generates only one Lyndon
word, i.e. $001$, whose matrix $M((001)^3)$ has homogeneous
vertical projections equal to $1$, not enough to reach the desired
value~$2$.

\begin{figure}
\begin{center}
\includegraphics[width = 8cm]{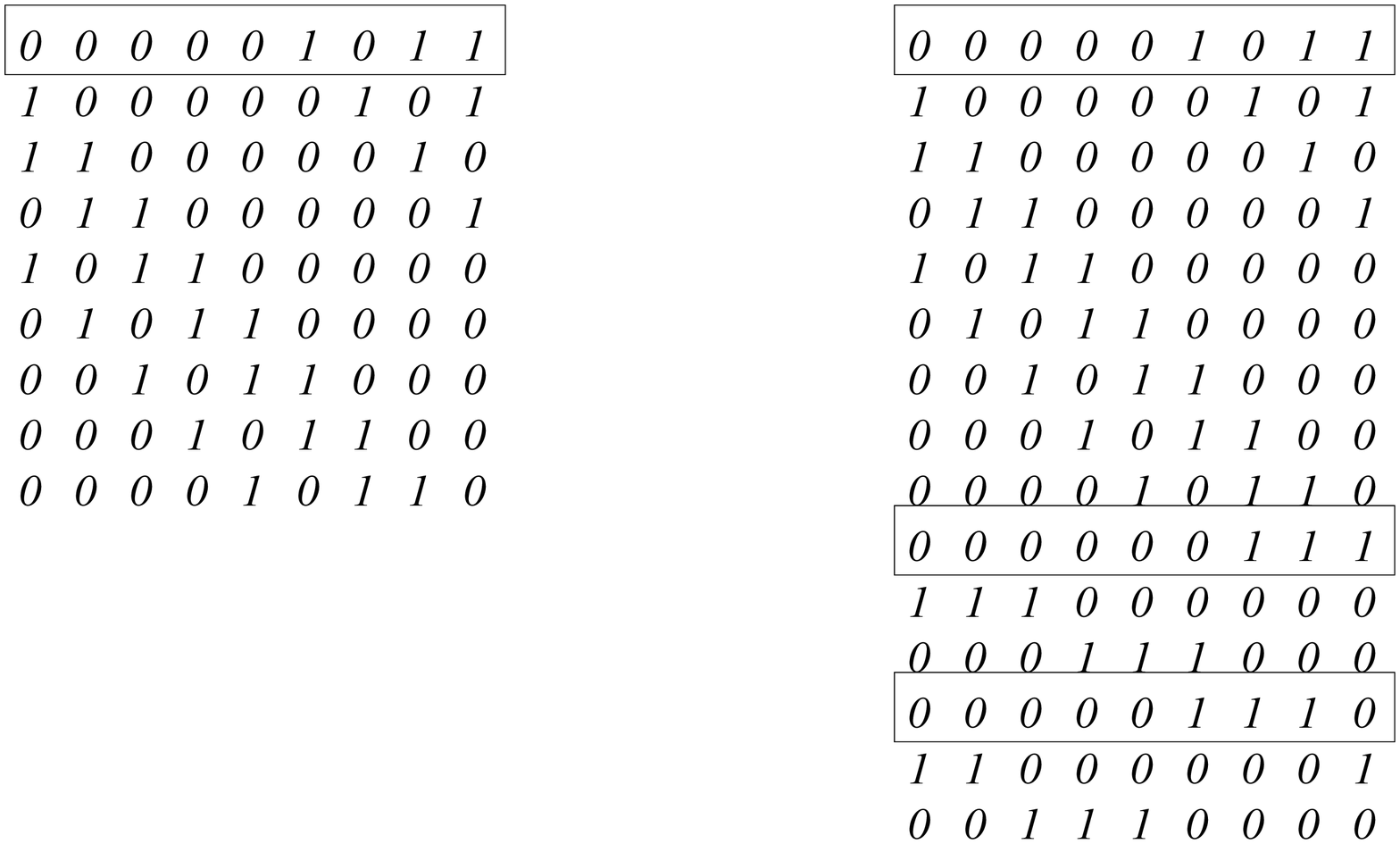}
\end{center}
\caption{The solution of dimension $15 \times 9$ obtained by
applying $Rec(H,V,{\mathcal{E}})$ when the horizontal projections
have constant value $3$, and the vertical projections $5$.}
\label{matrix3}
\end{figure}

\medskip

The validity of $Rec(H,V,{\mathcal{E}})$ is a simple consequence of
Theorem \ref{lyndon}. Clearly, the obtained matrix has homogeneous
horizontal and vertical projections, equal to $h$ and $v$,
respectively, and, by construction, all the rows are distinct.
Moreover, the algorithm always terminates since at each iteration,
we add as many rows as possible to the final solution. Concerning
the complexity analysis, we need to generate $O(m)$ different Lyndon
words and shift each of them $O(n)$ times. So, since the algorithm
{\bf GenLyndon$(n,h)$} requires  $O( k \cdot h \cdot n)$ steps to
generate $k$ words of $L(n,h)$, the whole process takes polynomial
time.

{\bf Remark:} Let us consider the special case where
$H=V=(h,\ldots,h)$ with $0<h<n$. The step 2.1 of Rec gives $q=1$
and $Rec(H,H,{\mathcal{E}})$ returns the matrix $A_0$ with the
first row $(0)^{n-h}(1)^h$. Hence we remark that $^tA_0=A_0$ and
so any two columns are different.

In graph $G$ a {\it twin} is a pair of vertices $\{u,v\}$ such
that $u$ and $v$ have the same neighborhood $N_G(u)=N_G(v)$.
Moreover a graph $G$ is {\it regular} if each vertex has same
number of neighbors. Using a straightforward counting argument we
have that if $G=(X,Y,E)$ is a bipartite regular graph  then $\vert
X\vert=\vert Y\vert$.   Hence if $G=(X,Y,E)$ is a bipartite
regular graph without twins its  incidence matrix $A_G$ satisfies:
(1) the horizontal and vertical projections satisfy $H=V$, (2)
both the horizontal and vertical rows are pairwise distinct.

The next result follows directly  from the algorithm {\bf
Rec$(H,V,{\mathcal E)}$} we designed above.

\begin{corollary}
Given $n$ and $k$ two positive integers, the construction of a
$k$-regular bipartite graph $G=(X,Y,E), \vert X\vert=\vert
Y\vert=n,$ without twins, if any, can be done in polynomial time.
Moreover the following condition  characterizes the degree
sequence of  a $k$-regular bipartite graphs without twins:
$d_i=k,0<k<n,$ for each vertex $v_i\in X\cup Y$.
\end{corollary}
\begin{proof}
It directly follows from the remark just above.
\end{proof}

\section{Reconstruction of an $h$-uniform hypergraph with span one}\label{spanone}
Let us consider  the case where the $h$-uniform hypergraph to
reconstruct is almost $d$-regular, in other words its degree
sequence has a span one, i.e.,  its vertical projections are
$V=(v,\dots,v,v-1,\dots,v-1)$. So, let us indicate with
$\mathcal{E}_1$, the set of matrices having different rows,
homogeneous horizontal projections, and vertical span one
projections. In order to solve this problem we will use the
algorithm $Rec(H,V,{\cal E})$ we designed in the previous section.

Again in \cite{bro}, it has been proved that also for span one
projections vectors, Conditions $1$ and $2$ are sufficient to
ensure the existence of a compatible matrix; again
%Now, we slightly change Condition $3$ as follows:
%
%\begin{description}
%\item{Condition $3'$:} if $Consistency(H,V,\mathcal{E}_1)$ has
%positive answer, then the maximum value of the vertical
%projections $v$ is strongly upper bounded by $h/n \cdot
%{{n}\choose{h}}$, i.e., $v\leq h/n \cdot {{n}\choose{h}}$,
%\end{description}
Condition $3$, formulated with ${\cal E}_1$ instead of $\cal E$,
succeeding in forcing that matrix to belong to the set
$\mathcal{E}_1$. So, let us consider the following algorithm that
relies on Rec$(H,V,{\mathcal E)}$:

\smallskip

\noindent{\bf RecSpan1$(H,V,{\mathcal E}_1)$}

\begin{description}
\item[Input :] Two vectors: $H=(h, \ldots ,h)$ of length $m$, and
$V=(v, \ldots ,v,v-1,\dots, v-1)$ of length $n$, satisfying
Conditions $1$, $2$, and $3$.

\item[Output :] An element $A_1$ of the class ${\mathcal E}_1$
having $H$ and $V$ as horizontal and vertical projections,
respectively.

\item[Step 1:] let $n_0$ and $n_1$ be the number of elements $v$
and $v-1$ of $V$, respectively, and set $k$ to be the least
integer such that it is both multiple of $h$ and $n$, and greater
than $h \cdot m$. Create the homogeneous vectors of projections
$H'$ and $V'$ such that $H'=(h,\dots,h)$ has length $m' = k/h >
m$, and $V'=(v',\dots,v')$, of length $n$ and $v'=k/n$.

\item[Step 2:] run Rec$(H',V',{\mathcal E})$, and let $A$ be its
output matrix.

\item[Step 3:] act on the submatrix $M_0(u)$ of $A$, as defined in
Proposition \ref{bar}, by deleting the rows $s^{i\:h}(u)$, with $0
\leq i < t$, and $t=(n(v' - v)+n_1)/ h$. Give the obtained matrix
$A_1$ as output, after rearranging the columns in order to obtain
the desired sequence of vertical projections.
\end{description}

Again a simple example will clarify the algorithm: we consider the
instance $H=(3,\dots,3)$ of length $13$, and
$V=(5,5,5,4,4,4,4,4,4)$. Step $1$ sets $n_0=3$, $n_1=6$, and
$k=45$ is the least integer multiple of $h=3$, $n=9$, and greater
than $3\cdot 13= 39$.

\begin{figure}
\begin{center}
\includegraphics[width = 7.5cm]{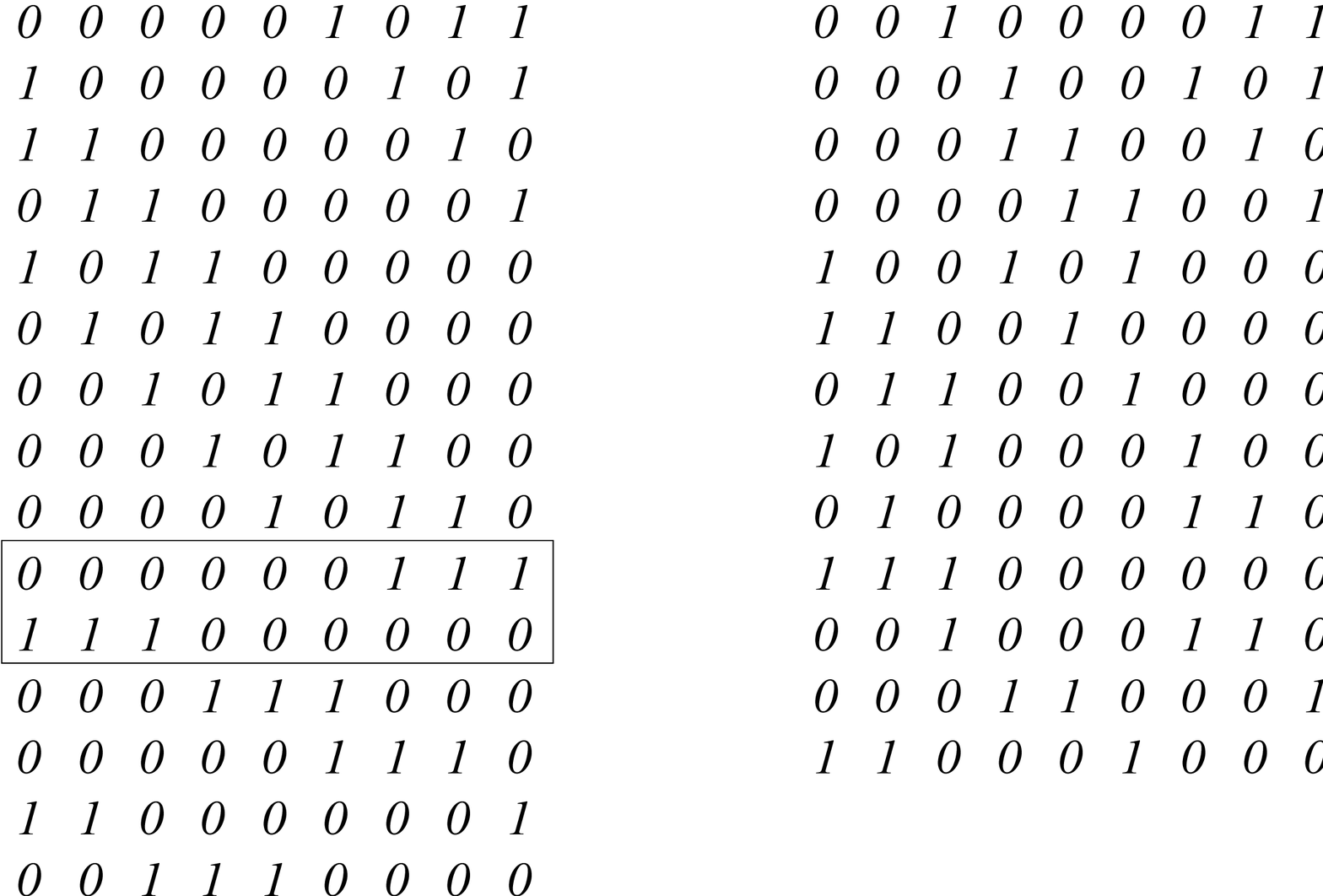}
\end{center}
\caption{On the left, two rows of $M_0(u)$ are deleted from the
matrix given on the right side of Fig.~\ref{matrix3}, which is the output of
$Rec(H',V',\mathcal{E})$. On the right, a rearrangement of its
columns makes it compatible with the initial sequence $V$.}
\label{matrix4}
\end{figure}

Step $2$ runs $Rec(H',V',\mathcal{E})$, on the homogeneous
instance $H'=(3,\dots,3)$ of length $15$, and $V'=(5,\dots,5)$ of
length $9$. The output matrix $A$ is given by Fig.~\ref{matrix3},
on the right.

Then, Step $3$ deletes the rows $s^0(u)=u=(0)^6(1)^3$, and
$s^3(u)=(1)^3(0)^6$ of the submatrix $M_0(u)$, being
$t=(9\cdot(5-5)+6)/3=2$, as in Fig.~\ref{matrix4}, on the left.

Finally a rearrangement of the columns is needed in order to make
the matrix compatible with the starting vector $V$, as in
Fig.~\ref{matrix4}, on the right. More precisely
columns $4$, $5$, and $6$ are shifted in the first three
positions, preserving their order.

{\bf Remark:} a rearrangement of the columns of a matrix causes a
related rearrangement of the elements of the vector of the
vertical projections, without modifying the values of its
elements. Furthermore, it is straightforward that such a
rearrangement also preserves the inequality relation between the rows.

\medskip

The correctness of RecSpan1$(H,V,{\mathcal E}_1)$ follows after
observing that:

\begin{description}

\item{$i)$} by definition of $k$, it holds: $$k/n - v <
h/gcd\{n,h\}.$$ In words, this means that the reconstructed matrix
$A$ compatible with the homogeneous vectors $H'$ and $V'$ is a
minimal one, w.r.t. the dimensions, including $A_1$. Furthermore,
the difference between the number of rows of $A_1$ and $A$ is less
than the rows of $M_0(u)$;

\item{$ii)$} in Step $2$, the vectors $H'$ and $V'$ satisfy
Condition $3$ by definition of $k$, and since $H$ and $V$ do. As a
consequence the call of Rec$(H',V',{\mathcal E})$ always
reconstructs a matrix $A$;

%{\color{red} state that if condition 3 is satisfied by H and V
%then it is also satisfied by H' and V', so we can use step 2}

\item{$iii)$} it is straightforward that the algorithm
Rec$(H',V',{\mathcal E})$ always inserts the submatrix $M_0(u)$,
with $u=(0)^{n-h}(1)^h$, in $A$. So, the deletion of the rows
$s(u)^{i\:h}$, according to the consecutive values of $i$ ranging
from $1$ to $t$, forces the vertical projections to maintain two
different consecutive values at each time, and reaching the
desired values.
\end{description}

It is straightforward that the complexity of RecSpan1$(H,V,{\mathcal
E}_1)$ is the same as  Rec$(H,V,{\mathcal E})$.

\section{Conclusion}
The question of necessary and sufficient conditions for the
existence of a simple hypergraph ${\cal H}=(Vert,{\cal E}), \vert
Vert\vert = n,\vert{\cal E}\vert=m,$ with a given degree sequence
is a long outstanding open question even in the case of a
$3$-uniform hypergraph ($\vert e\vert=3$ for each $e\in \cal E$).
In this paper, we answered to this question in the special case
where $\cal H$ is $h$-uniform and $d$-regular or $\cal H$ is
$h$-uniform and almost $d$-regular, i.e. the degree sequence of
$Vert$ is $(d_1=v,d_2=v,\ldots,d_n=v)$,
$(d_1=v,\ldots,d_{n_0}=v,d_{n_0+1}=v-1,\ldots,d_{n_0+n_1}=v-1)$,
respectively. Merging the results of  the three previous sections
we can state the following:

\begin{theorem}\label{charact}
${\cal H}=(Vert,{\cal E}), \vert Vert\vert = n,\vert{\cal
E}\vert=m,$ is a $h$-uniform $d$-regular, respectively $h$-uniform
almost $d$-regular, hypergraph if and only if
\begin{enumerate}
\item $mh=nv$, resp. $mh=nv-n_1$;
\item $h\le n,v\le m$;
\item $v\leq h/n \cdot{{n}\choose{h}}$.
\end{enumerate}

\end{theorem}

Moreover, given a degree sequence satisfying the conditions of
this theorem, we give two linear time (in the size of the
incidence matrix) algorithms that construct a $h$-uniform
$d$-regular hypergraph or a $h$-uniform almost $d$-regular
hypergraph.

A next step to the characterization of the degree sequence of a
simple hypergraph would be its study for  the subclass of uniform
hypergraphs (in particular three uniform hypergraphs) with span $k$, i.e. the degree of any vertex ranges
from $\{v-k,v-k+1,\ldots,v\}$ a set of $k$ successive values,
where $k\geq 2$ is a fixed integer.

\end{document}